\documentclass[runningheads,fullpage,11pt]{llncs}
\usepackage{amsfonts}
\usepackage{algorithm}
\usepackage{algorithmic}
\usepackage{amssymb}
\usepackage{eepic,graphics}
\usepackage{epsfig}
\newcommand{\proofendsign}{$\Box$}
\renewenvironment{proof}{{\noindent \bf Proof }}{{\hspace*{\fill}\proofendsign\par\bigskip}}
\newcommand{\cbw}{{\bf cut\textrm{-}bool}}
\newcommand{\bw}{{\bf boolw}}
\newcommand{\ccw}{{\bf cut\textrm{-}car}}
\newcommand{\cw}{{\bf carw}}
\newcommand{\tw}{{\bf tw}}
\newcommand{\vc}{{\bf VC}}

\begin{document}
\pagestyle{headings}  
\title{On the boolean-width of a graph: structure and applications}

\author{Yuri Rabinovich\inst{1} \and J. A. Telle\inst{2}\thanks{Supported by the Norwegian Research Council, project PARALGO.}}
\institute{
Department of Computer Science, Haifa University, Israel \and
Department of Informatics, University of Bergen, Norway\\
\texttt{yuri@cs.haifa.ac.il, telle@ii.uib.no}
}
\maketitle
\begin{abstract}
We study the recently introduced boolean-width of graphs.
Our structural results are as follows.
Firstly, we show that almost surely the boolean-width of a random
graph on $n$ vertices is $O(\log^2 n)$, and it is easy to find the corresponding decomposition tree.

Secondly, for any constant $d$ a graph of maximum degree $d$ has boolean-width linear in treewidth.
This implies that almost surely the boolean-width of a (sparse) random
$d$-regular graph on $n$ vertices is linear in $n$.

Thirdly,  we show that the boolean-cut value is well approximated by VC dimension
of corresponding set system. Since VC dimension is widely studied, we hope that this structural result will prove helpful in better understanding of boolean-width.

Combining our first structural result with algorithms from Bui-Xuan et al \cite{BTV09,BTV09II} we get
for random graphs
quasi-polynomial $O^*(2^{O(\log ^4 n)})$ time algorithms for a large class of vertex subset and vertex partitioning problems.
\end{abstract}
%
%
\section{Introduction}

Width parameters of graphs, like tree-width and clique-width, 
 are important in the theory of graph algorithms, see e.g. \cite{HOSG08}.
Recently, Bui-Xuan, Telle and Vatshelle \cite{BTV09} introduced a new width parameter of graphs called boolean-width. It is our feeling that this parameter  has interesting structural properties and is important for applications, e.g. its value for any graph is no larger than branch-width or clique-width
and if a decomposition of small boolean-width $k$ is given then we get fast algorithms for a large class of problems. Such problems include Minimum Dominating Set and Maximum Independent Set 
solved in time $O(n(n + 2^{3k} k))$ \cite{BTV09},
and also a large class of vertex subset and vertex partitioning problems in time
$O^*(2^{d \times q \times k^2})$, for problem-specific constants
$d$ and $q$ \cite{BTV09II}.
This paper is dedicated to a study of boolean-width.

We show that asymptotically almost surely the boolean-width of
a random graph on $n$ vertices is $O(log^2 n)$, for any decomposition tree.
This implies quasi-polynomial time $O^*(2^{O(\log ^4 n)})$ algorithms for the above mentioned problems on random graphs.
This result contrasts sharply with a recent result of
Mare\v{c}ek \cite{M09} who showed that almost surely 
a random graph on $n$ vertices has rank-width linear in $n$. The latter parameter was introduced by Oum and Seymour and is smaller than tree-width+1, branch-width, clique-width and NLC-width.

For sparse random graphs we show a complementary result, namely that for any constant $d$ almost surely the boolean-width of a random
$d$-regular graph on $n$ vertices is linear in $n$.
This as corollary to another result, that any  graph of maximum degree $d$ has boolean-width linear in treewidth.

Finally,  we show that the boolean-cut value is well approximated by VC dimension
of corresponding set system. Since VC dimension is widely studied and relatively well understood, we hope that this structural result will prove helpful in future studies.

\section{Definitions}

We consider undirected graphs without loops and denote the vertex set of a graph $G$ by $V(G)$
and the neighbors of a vertex $v$ by $N(v)$.
For $A\subseteq V(G)$ we let $\overline{A}$ denote the set $ V(G) \setminus A$
and let $N(A) \subseteq \overline{A}$ denote the neighbors of $A$ in $\overline{A}$.
The following formalism is standard in graph and matroid decompositions (see, {\sl e.g.},~\cite{RS91}).
\begin{definition}\label{def_f_width}
A decomposition tree of a graph $G$ is a pair $(T,\delta)$ where $T$ is a tree having internal nodes of degree three and $n=|V(G)|$ leaves, and $\delta$ is a bijection between the vertices of $G$ and the leaves of $T$.
Every edge of $T$ defines a cut $\{A,\overline{A}\}$ of the graph, i.e. a partition of $V(G)$ in the two parts given, via $\delta$, by the leaves of the 
two subtrees of $T$ we get by removing the edge.
Let $f:2^V\rightarrow\mathbb{R}$ be a symmetric function, i.e. $f(A)=f(\overline{A})$ for all $A\subseteq V(G)$, also called a \emph{cut function}.
The $f$-width of $(T,\delta)$ is the maximum value of $f(A)$, taken over all
cuts $\{A, \overline A\}$ of $G$ given by an edge $uv$ of $T$.
The $f$-width of $G$ is the minimum $f$-width over all decomposition trees of $G$.
\end{definition}

\begin{definition}[Boolean-width]
\label{def:booleanWidth}
The $\cbw:2^{V(G)}\rightarrow\mathbb{R}$ function of a graph $G$ is defined as
$$\cbw(A)=\log_2|\{S\subseteq \overline{A}:~\exists X\subseteq A ~~ \wedge ~~
S=N(X)\}|$$
It is known from boolean matrix theory that $\cbw$ is symmetric~\cite[Theorem 1.2.3]{K82}.
Using Definition~\ref{def_f_width} with $f=\cbw$ we define the boolean-width of 
a decomposition tree, denoted $\bw(T,\delta)$, and
the boolean-width of a graph, denoted $\bw(G)$.
\end{definition}
 
For a vertex subset $A$, the value of $\cbw(A)$ can also be seen as the logarithm in base $2$ of the number of pairwise different vectors that are spanned, via boolean sums (1+1=1), by the rows  of the
$A \times \overline{A}$ sub-matrix of the adjacency matrix of $G$.

\section{Structural Results}

\subsection{Random graphs}

Let $G_p$ be a random graph on $n$ vertices where each edge is chosen randomly and independently 
with probability $p$. 
\begin{theorem}
 \label{th:bw}
Almost surely, $\bw(G_p)\,=\, O\left( \frac{\ln^2 n}{p} \right)$.
\end{theorem}
We prove first the following lemma. 
\begin{lemma}
 \label{cl:k-sets}
Let $G_p$ be a graph as above, and let $k_p=\frac{2\ln n}{p}$.
Then, almost surely, for {\em all} subsets of vertices $S \subset V(G)$ with $|S|=k_p$ it holds that
$|N(S)| \geq |\overline{S}| - k_p$. 
\end{lemma}
\begin{proof}
In what follows, we write simply $G$ and $k$. 
Fix a particular $S$ with $|S|=k$. For every $v\in \overline{S}$, let $X_v$ be 1 if $v\not\in N(S)$, 
and $0$ otherwise. Clearly, $X_v=1$ with probability $(1-p)^k$, and 
$\sum_{v\in \overline{S}} X_v =  |\overline{S}|-|N(S)|$. 
Observe that $E[\sum_{v\not\in S} X_v] \;=\; (1-p)^k (n -k) \;<\; (1-p)^k n$. Call this expectation $\mu$.
By Chernoff Bound (see e.g. \cite{MR95}, p.68), 
\[
 \Pr\left[ \sum_{v\in \overline{S}} X_v \geq k\right] ~<~ \left({\frac{e\mu}{k}}\right)^k  ~<~  
\left( (1-p)^{k} n \right)^k  ~=~ \left( (1-p)^{2\ln n/p} n \right)^k ~<~ n^{-k}\,,  
\]
the last inequality due to the fact that for $p\in (0,1)$,  $(1-p)^{\frac{1}{p}} \leq e^{-1}$.

Applying the union bound, we conclude that the probability that there exists $S$ of size $k$ such that
$|N(S)| < |\overline{S}| - k$ is at most
\[
 {n \choose k} \cdot n^{-k} ~<~ (k!)^{-1} ~=~ o(1)\,
\]
and the statement follows.
\end{proof}
\begin{corollary}
 \label{cor:random_cut}
For $G=G_p$ and $k=k_p$ as before, for {\em all} cuts $\left\{ A,\overline{A} \right\}$ in
$G$ it holds almost surely that $\cbw(A) \geq O\left( \frac{\ln^2 n}{p} \right)$. 
\end{corollary}
\begin{proof}
The number of distinct sets $N(S) \cap \overline{A}$ contributed by the sets $S \subseteq A$ with $|S| \leq k$
is at most $\sum_{i=0}^k {n \choose i}$.  By the previous lemma, for all sets
$S \subseteq A$ with $|S| \geq k$, it holds almost surely that $|N(S) \cap \overline{A}| \geq |\overline{A}| - k$.
Therefore, almost surely, the sets $S \subseteq A$ with $|S| \geq k$, also contribute at most 
$\sum_{i=0}^k {n \choose i}$ distinct sets $N(S) \cap \overline{A}$. Thus, almost surely there are 
at most $2\sum_{i=0}^k {n \choose i}$ distinct sets $N(S) \cap \overline{A}$ altogether. Taking the 
logarithm, we arrive at the desired conclusion.
\end{proof}
The Theorem~\ref{th:bw} now follows easily: for {\em any} decomposition tree, all the
cuts it defines will almost surely have a cut-Boolean-width at most $O\left(\frac{\ln^2 n}{p} \right)$.

\subsection{Graphs of bounded degree}

\begin{theorem}
\label{th:regular}
For $G$ a graph of maximum degree $d$ we have 
\[
 \frac{1}{6d^2}\cdot\tw(G) ~\leq~ \bw(G) ~\leq~ \tw(G)+1\,.
\]
\end{theorem}
\begin{proof}
The second inequality is known to hold for any graph \cite{ABRV09}.
Using Definition~\ref{def_f_width} with $f=\ccw$ counting the number of edges in the cut $\{A, \overline{A}\}$,
we define the carving-width of a decomposition tree, and
the carving-width of a graph, denoted $\cw(G)$. This graph-theoretic parameter has been studied previously,
and in particular it is known that $3 \cdot \cw(G) \geq \tw(G)$~\cite{TSB00}.

Consider now our graph $G$. Observe that for any $A \subseteq V(G)$ there exists 
$S \subseteq A$ and $S' \subseteq \overline{A}$ such that $|S|=|S'|\geq \ccw(A)/(2d^2)$, each vertex $v$ in $S$ has 
a single neighbour $u$ in $S'$, and, moreover no other $v'\in S$ sees $u$, i.e., $N(u) \cap S = v$.
Indeed, start with the original cut $\{A,\overline{A}$, and as long as it is not empty, do the following.
Pick an edge $(v,u)$ in the cut, where $v\in A$ and $u\in \overline{A}$, add $v$ to $S$ and $u$ to $\overline{S}$,
and remove the vertices in $N(v) \cap \overline{A}$ and $N(u) \cap A$ together with incident edges. Since
in each iteration at most $2d^2$ edges are removed, there will be at least $\ccw(A)/(2d^2)$ iterations,
and therefore the size of $S$ will be as claimed. 

The existence of such $S,S'$ at once imply that 
\[
 |\left\{ N(K) \cap \overline{A}, ~K~ \subseteq A \right\}| ~\geq~ 2^{|S|}\,,
\]
as each subset $K \subseteq S$ has a distinct $N(K) \cap \overline{A}$.

To sum up, $2d^2\cbw(A) \geq \ccw(A)$.

Consider now the decomposition tree yielding the $\bw(G)$. Using the above observation, we conclude
that the maximal $\ccw(A)$ of any involved cut is at most $2d^2\bw(G)$, and thus $2d^2\bw(G) \geq \cw(G)$.
Recalling that $\cw(G) \geq {1 \over 3}\tw(G)$, the conclusion follows.
\end{proof}
As an easy corollary of this theorem we get the following result, which can be viewed as a counterpart of Theorem~\ref{th:bw} for sparse random graphs.
\begin{theorem}
 \label{th:bw2}
Let $G$ be a random $d$-regular graph on $n$. Then, almost surely, $\bw(G) = \Omega({n \over d})$.
\end{theorem}
\begin{proof}
It is well known that every balanced cut (in our case, up to $({1 \over 3},{2 \over 3})$) in a random 
$d$-regular graph contains $\Omega(dn)$ edges. Consider the decomposition tree yielding the $\bw(G)$.  
By a standard argument, one of the involved cuts $\{A,\overline{A}\}$ must be $({1 \over 3},{2 \over 3})$-balanced,
and hence almost surely  $\ccw(A) =  \Omega(dn)$. Since, as we have seen in the proof of 
Theorem~\ref{th:regular}, $\cbw(A) \geq \ccw(A)/(2d^2)$, the statement follows.
\end{proof}

\subsection{Boolean cut-width and the VC-dimension of the family of neighbourhoods}
In this section we show that the Boolean cut-width of a cut is closely related to the 
VC-dimension of the related system of neighbourhoods, which in turn has a simple 
description in our case.

Let $\{A, \overline{A} \}$ be a cut in $G$, and let $M$ be the corresponding $|A|\times |\overline{A}|$
Boolean adjacency matrix. Consider the family ${\cal N} =  \{ N(K) \cap \overline{A} : ~K~ \subseteq A \}$. 
The Vapnik-Chervonenkis dimension of $\cal N$ is defined as the maximum size of $T \subseteq \overline{A}$
that is {\em shattered} by ${\cal N}$, i.e., any subset $T' \subseteq T$ is of the form  $N(K) \cap T$ for
some $K \subseteq A$. We shall denote this dimension as $\vc(A)$.
Observe that since ${\cal N}$ is closed under unions, $\vc(A)$  is just the size of the maximum permutation submatrix of $M$.   
\begin{theorem}
 \label{th:VC}
\[
 \vc(A) ~\leq~ \cbw(A) ~\leq~ \log n \cdot \vc(A) 
\]
\end{theorem}
\begin{proof}
The first inequality is obvious, since the existence of permutation submatrix of size $k$ in $M$ implies that $|{\cal N}| \geq 2^k$. It was already expoited in the proof of Theorem~\ref{th:regular}.
For the second inequality, we use the following fundamental lemma, variously attributed to Sauer, to Perles and Shelah, as well as to Vapnik and Chervonenkis:
\begin{lemma} {\bf[S,PS,VC Lemma]}
Let  ${\cal F}$ be a family of subsets of some underlying set of size $n$. Then, if 
$|{\cal F}| \geq \sum_{i=0}^k {n \choose i}$, then $\vc({\cal F}) > k$.
\end{lemma}
Since $\log_2 \sum_{i=0}^k {n \choose i} \;>\; k\log_2 (n/k)$, we conclude that 
\[
 \vc(A) ~=~ \vc({\cal N}) ~>~ {{\log_2 |{\cal N}}| \over {\log_2 n}} ~=~ {{\cbw(A)}\over {\log_2 n}}\;.
\]
\end{proof}

\section{Algorithmic Results}
In this section, again we let $G_p$ be a random graph on $n$ vertices where each edge is chosen randomly and independently 
with probability $p$. 
Notice that the proof of Theorem \ref{th:bw} actually yields the stronger result:

\begin{theorem}
\label{th:bw3}
Almost surely, any decomposition tree of $G_p$ has boolean-width $O\left( \frac{\ln^2 n}{p} \right)$.
\end{theorem}
Theorem \ref{th:bw3} implies that a large class of vertex subset and vertex partitioning problems,
can be solved in quasi-polynomial time on random graphs. 
In this section we provide two corollaries to this effect.

\begin{definition}\label{def_sigma_rho}
\emph{Let $\sigma$ and $\rho$ be finite or co-finite subsets of natural numbers. 
A subset $X$ of vertices of a graph $G$ is a {\sl $(\sigma,\rho)$-set} of $G$ if
$$~\hfill\forall v \in V(G): |N(v) \cap X| \in  \left\{ 
   \begin{array}{ll}
    \sigma & \mbox{if $v \in X$}, \\
    \rho & \mbox{if $v \in V(G) \setminus X$}.
   \end{array} \right.\hfill~$$
Let $d(\mathbb{N}) = 0$.
For every finite or co-finite set $\mu \subseteq \mathbb{N}$, let $$d(\mu) = 1+min \{ max \{x : x \in \mu\}, max \{x: x \notin \mu\} \}$$
Let $d(\sigma,\rho)=max\{d(\sigma),d(\rho)\}$.}
\end{definition}

The $(\sigma,\rho)$ {\sl vertex subset problems} consist of finding the size of a minimum or maximum $(\sigma$,$\rho)$-set in $G$.

\begin{theorem} \label{thm_vs} [\cite{BTV09II}]
For every $n$-vertex, $m$-edge graph $G$ given along with a decomposition tree of boolean-width $k$,
 any minimum or maximum $(\sigma,\rho)$ vertex subset problem on $G$ can be solved in $O(n(m+d \cdot k 2^{3d \cdot k^2 + k}))$ time,
 where $d$ stands for $d(\sigma,\rho)$.
\end{theorem}

\begin{corollary}
Any minimum or maximum $(\sigma,\rho)$ vertex subset problem on $G_p$ can be solved in $O^*(2^{O(d(\sigma,\rho) \cdot \log ^4 n)})$ time.
\end{corollary}
Several NP-hard problems are expressible in this framework, {\sl e.g.} problems like
Max Independent Set (with $\sigma=\{0\}$, $\rho=\mathbb{N}$, $d(\sigma,\rho)=1$) and
Min Dominating Set ($\sigma=\mathbb{N}$, $\rho=\{1,2, \dots \}$, $d(\sigma,\rho)=2$).
Also problems like
Min $p$-Dominating Set ($\sigma=\mathbb{N}$, $\rho=\{p,p+1, \dots \}$, $d(\sigma,\rho)=p+1$) and
Max Induced $p$-Bounded Degree Subgraph ($\sigma=\{ 0,1, \dots p\}$, $\rho=\mathbb{N}$, $d(\sigma,\rho)=p+1$).

The framework is extendible to problems asking for a partition of $V(G)$ into $q$ classes,
with each class satisfying a certain $(\sigma,\rho)$-property, as follows.


\begin{definition}
\label{def1}
\emph{Let $D_q$ be a $q$ by $q$ matrix with entries being 
finite or co-finite subsets of natural numbers. 
A {\sl $D_q$-partition} in a graph $G$ is a partition $\{V_1,V_2,...,V_q\}$ of 
$V(G)$ such that 
for $1 \leq i,j \leq q$ we have 
$\forall v \in V_i:
|N(v) \cap V_j| \in D_q[i,j]$.
Let $d(D_q)=\max_{i,j}d(D_q[i,j])$.}
\end{definition}

The {\sl vertex partitioning problems}  consist of deciding if $G$ has
a $D_q$ partition, the so-called $D_q$-problem.
NP-hard problems fitting into this framework include {\sl e.g.} for any fixed graph $H$
the problems known as $H$-Coloring or $H$-Homomorphism (with $K_q$-Coloring deciding if chromatic number is at most $k$), $H$-Covering, $H$-Partial Covering, and in general the question of deciding if an input
graph has a partition into $q$ $(\sigma, \rho)$-sets.

\begin{theorem} \label{th:2} [\cite{BTV09II}]
\label{theo_vertex_part}
For every $n$-vertex, $m$-edge graph $G$ given along with a decomposition tree  having boolean-width $k$,
any $D_q$-problem on $G$ can be solved in
$O(n(m+qd k 2^{3qd k^2 + k}))$ time, where $d$ stands for $d(D_q)$.
\end{theorem}

\begin{corollary}
Any $D_q$-problem on $G_p$ can be solved in 
 $O^*(2^{O(q \cdot d(D_q) \times \log ^4 n)})$ time.
\end{corollary} 

Simple extensions will allow also to solve weighted versions and search versions of both the vertex subset and vertex partitioning problems.





\end{document}